\documentclass[10pt,a4paper]{amsart}
\usepackage{amssymb,amsmath,amsfonts}

\headsep=1.2500cm \numberwithin{equation}{section}
\newcommand{\set}[1]{\left\{#1\right\}}
\newcommand{\A}{UP(I,\mathcal{A})}
\newtheorem{Theorem}{Theorem}[section]
\newtheorem{Proposition}[Theorem]{Proposition}
\newtheorem{cor}[Theorem]{Corollary}
\newtheorem{lemma}[Theorem]{Lemma}
\theoremstyle{remark}
\newtheorem{Definition}[Theorem]{Definition}
\newtheorem{Example}[Theorem]{Example}

\newtheorem{Remark}[Theorem]{Remark}

\begin{document}
\title{strong pseudo-Connes amenability of dual Banach algebras}

\author[S. F. Shariati]{S. F. Shariati}

\address{Faculty of Mathematics and Computer Science,
	Amirkabir University of Technology, 424 Hafez Avenue, 15914
	Tehran, Iran.}

\email{f.shariati@aut.ac.ir}

\author[A. Pourabbas]{A. Pourabbas}
\email{arpabbas@aut.ac.ir}

\author[A. Sahami]{A. Sahami}
\address{Department of Mathematics Faculty of Basic Science, Ilam University, P.O. Box 69315-516 Ilam, Iran.}
\email{a.sahami@ilam.ac.ir}



\begin{abstract}
In this paper, we introduce the new notion of strong pseudo-Connes amenability for dual Banach algebras. We study the relation between this new notion to the various notions of Connes amenability. Also we show that for every non-empty set $I$, $M_I(\mathbb{C})$ is strong pseudo-Connes amenable if and only if $I$ is finite. We provide some examples of certain dual Banach algebras and we study its strong pseudo-Connes amenability. 

In the last section, we investigate the property ultra central approximate identity for a Banach algebra $\mathcal{A}$ and its second dual $\mathcal{A}^{**}$. Also we show that for a left cancellative regular semigroup $S$, ${\ell^{1}(S)}^{**}$ has an ultra
 central approximat identity if and only if $S$ is a group. Finally we study this property for $\varphi$-Lau product Banach algebras and the module extension Banach algebras.
\end{abstract}
\maketitle
\section{Introduction and Preliminaries}
The concept of amenability for Banach algebras were first introduced by B. E. Johnson \cite{Runde:2002}. A Banach algebra $\mathcal{A}$ is amenable if and only if there exists a bounded net ($m_{\alpha}$) in ${\mathcal{A}}\hat{\otimes}{\mathcal{A}}$ such that $a\cdot m_{\alpha}-m_{\alpha}\cdot a{\rightarrow}0$ and $\pi_{\mathcal{A}}(m_{\alpha})a{\rightarrow}a$ for every $a\in{\mathcal{A}}$. By removing the boundedness condition in the definition of amenability, Ghahramani and Zhang introduced the notion of pseudo amenability \cite{ghah:07}. Recently Sahami introduced a notion of amenability, named strong pseudo-amenability \cite{Sah:18}. A Banach algebra $\mathcal{A}$ is called strong pseudo-amenable, if there exists a (not necessarily bounded) net ($m_{\alpha}$) in $({\mathcal{A}}\hat{\otimes}{\mathcal{A}})^{\ast\ast}$ such that
\begin{equation*}
a\cdot m_{\alpha}-m_{\alpha}\cdot a\rightarrow0,\quad a\pi^{**}_{\mathcal{A}}(m_{\alpha})=\pi^{**}_{\mathcal{A}}(m_{\alpha})a\rightarrow a\quad (a\in {\mathcal{A}}).
\end{equation*}
The class of dual Banach algebras were introduced by Runde \cite{Runde:2001}. Let $\mathcal{A}$ be a Banach algebra and let $E$ be a Banach $\mathcal{A}$-bimodule. An $\mathcal{A}$-bimodule $E$ is called dual if there is a closed submodule ${E}_{\ast}$ of ${E}^{\ast}$ such that $E=(E_{\ast})^{\ast}$. The Banach algebra $\mathcal{A}$ is called dual if it is dual as a Banach $\mathcal{A}$-bimodule. For a given dual Banach algebra $\mathcal{A}$ and a Banach $\mathcal{A}$-bimodule $E$, we denote by $\sigma{wc}(E)$, the set of all elements $x\in{E}$ such that the module maps $\mathcal{A}\rightarrow{E}$; ${a}\mapsto{a}\cdot{x}$ and ${a}\mapsto{x}\cdot{a}$
are $wk^\ast$-$wk$-continuous, which is a closed submodule of $E$. Since $\sigma{wc}(\mathcal{A}_{\ast})=\mathcal{A}_{\ast}$, the adjoint of $\pi$ maps $\mathcal{A}_{\ast}$ into $\sigma{wc}(\mathcal{A}\hat{\otimes}\mathcal{A})^{\ast}$. Therefore, $\pi^{\ast\ast}$ drops to an $\mathcal{A}$-bimodule morphism $\pi_{\sigma{wc}}:(\sigma{wc}(\mathcal{A}\hat{\otimes}\mathcal{A})^{\ast})^{\ast}\rightarrow\mathcal{A}$. A suitable concept of amenability for dual Banach algebras is the Connes amenability. This notion under different name, for the first time was introduced by Johnson, Kadison, and Ringrose for von Neumann algebras \cite{John:72}. The concept of Connes amenability for the larger class of dual Banach algebras was later extended by Runde \cite{Runde:2001}. A dual Banach algebra $\mathcal{A}$ is called Connes amenable if and only if $\mathcal{A}$ has a $\sigma{wc}$-virtual diagonal, that is, there exists an element $M\in{(\sigma{wc}(\mathcal{A}\hat{\otimes}\mathcal{A})^{*})^*}$ such that	$a\cdot{M}=M\cdot{a}$ and ${a}\pi_{\sigma{wc}}(M)=a$ for every $a\in{\mathcal{A}}$ \cite{Runde:2004}.

Mahmoodi introduced and studied the notion of pseudo-Connes amenability for dual Banach algebras \cite{Mahmoodi:14}. A dual Banach algebra $\mathcal{A}$ is called pseudo-Connes amenable if there exists a net $(m_{\alpha})$ in $\mathcal{A}\hat{\otimes}\mathcal{A}$ such that for every $a\in{\mathcal{A}}$, $a\cdot m_{\alpha}-m_{\alpha}\cdot a\overset{wk^*}{\rightarrow}0$ in $(\sigma{wc}(\mathcal{A}\hat{\otimes}\mathcal{A})^{\ast})^{\ast}$ and $a\pi_{\sigma{wc}}(m_{\alpha})\overset{wk^*}{\rightarrow}a$ in $\mathcal{A}$ \cite[Definition 4.1]{Mahmoodi:14}. He showed that in
general the concepts of pseudo-Connes amenability and Connes amenability are distinct. Indeed he showed that $\ell^1(\mathbb{N},\max)$ is pseudo-Connes amenable but it is not Connes amenable \cite[Example 6.1]{Mahmoodi:14}.

By consideration these notions, we introduce a modification notion of pseudo-Connes amenability for dual Banach algebra which is weaker than Connes amenable and stronger than pseudo-Connes amenable:
\begin{Definition}\label{d1}
	A dual Banach algebra $\mathcal{A}$ is called strong pseudo-Connes amenable, if there exists a (not necessarily bounded) net ($m_{\alpha}$) in $(\sigma{wc}({\mathcal{A}}\hat{\otimes}{\mathcal{A}})^{\ast})^{\ast}$ such that for every $a\in{\mathcal{A}}$,
	\begin{enumerate}
		\item [(i)] $a\cdot m_{\alpha}-m_{\alpha}\cdot a\overset{wk^*}{\longrightarrow}0$ in $(\sigma{wc}(\mathcal{A}\hat{\otimes}\mathcal{A})^{\ast})^*$, and
		\item [(ii)] $a\pi_{\sigma{wc}}(m_{\alpha})=\pi_{\sigma{wc}}(m_{\alpha})a\overset{wk^*}{\longrightarrow}a$ in $\mathcal{A}$.
	\end{enumerate}
\end{Definition}
In section 2, we investigate the relation between this new notion to the various notions of Connes amenability. In section 3, we show that for the Banach algebra of $I\times I$-matrices over $\mathbb{C}$, with finite $\ell^1$-norm and matrix multiplication, $M_{I}(\mathbb{C})$ is strong pseudo-Connes amenable if and only if $I$ is finite. Also we show that for the set of all $I\times{I}$-upper triangular matrices, $\A$ under this new notion $I$ must be singleton. In section 4 we provide some examples of certain dual Banach algebras and we study its strong pseudo-Connes amenability. In particular, we show that $\ell^1(\mathbb{N},\max)$ is strong pseudo-Connes amenable but it is not Connes amenable.

Recently Sahami introduced the notion of ultra central approximate identity for Banach algebras
which is a generalization of the bounded approximate identity and the central approximate identity \cite{Sahami:18}. We say that A Banach algebra $\mathcal{A}$ has an ultra central approximate identity if
there exists a net $(e_\alpha)$ in $\mathcal{A}^{**}$ such that $ae_{\alpha}=e_\alpha a$ and $e_\alpha a\rightarrow a$; for every $a\in{\mathcal{A}}$. He
characterized the existence of an ultra central approximate identity for the semigroup algebra $\ell^1(S)$,
where $S$ is a uniformly locally finite inverse semigroup. Also he showed that for the Brandt semigroup
$S = M_0(G;I)$ over a non-empty set $I$, $\ell^1(S)$ has an ultra central approximate identity if
and only if $I$ is finite \cite[Theorem 2.8]{Sahami:18}. \\
In section 5 we study the relation between the property ultra central approximate identity and the notion $\varphi$-inner amenability for a Banach algebra $\mathcal{A}$ and $\mathcal{A}^{**}$. Also we show that for a left cancellative regular semigroup $S$, ${\ell^{1}(S)}^{**}$ has an ultra central approximat identity if and only if $S$ is a group. Finally as an application, we investigate the property of ultra central approximate identity for $\varphi$-Lau product Banach algebras and the module extension Banach algebras.
\section{Strong pseudo-Connes amenability}
\begin{Remark}\label{R1}
	\cite[Remark 2.1]{Sha:17} Let $\mathcal{A}$ be a dual Banach algebra and let $E$ be a Banach $\mathcal{A}$-bimodule. Since $\sigma wc(E^*)$ is a closed $\mathcal{A}$-submodule of $E^{*}$, we have a quotient map $q:E^{\ast\ast}\longrightarrow\sigma{wc}(E^{\ast})^{\ast}$ is defined by $q(M)=M\vert_{\sigma wc(E^*)}$ for every $M\in{E^{**}}$.
\end{Remark}
\begin{Proposition}\label{p2.2}
Let $\mathcal{A}$ be a strong pseudo-Connes amenable dual Banach algebra. Then $\mathcal{A}$ is pseudo-Connes amenable.
\end{Proposition}
\begin{proof}
Since $\mathcal{A}$ is strong pseudo-Connes amenable, there exists a net $(m_{\alpha})_{\alpha\in{I}}$ in $(\sigma{wc}({\mathcal{A}}\hat{\otimes}{\mathcal{A}})^{\ast})^{\ast}$ such that $a\cdot m_{\alpha}-m_{\alpha}\cdot a\overset{wk^*}{\longrightarrow}0$ and
 $a\pi_{\sigma{wc}}(m_{\alpha})=\pi_{\sigma{wc}}(m_{\alpha})a\overset{wk^*}{\longrightarrow}a$ for every $a\in{\mathcal{A}}$. Consider the quotient map $q:(\mathcal{A}\hat{\otimes}\mathcal{A})^{\ast\ast}\longrightarrow(\sigma{wc}({\mathcal{A}}\hat{\otimes}{\mathcal{A}})^{\ast})^{\ast}$ as in Remark \ref{R1}. Composing the canonical inclusion
 map $\mathcal{A}\hat{\otimes}\mathcal{A}\hookrightarrow(\mathcal{A}\hat{\otimes}\mathcal{A})^{\ast\ast}$ with $q$, by Goldstein's theorem we obtain a continuous $\mathcal{A}$-bimodule map $\eta:\mathcal{A}\hat{\otimes}\mathcal{A}\longrightarrow(\sigma\omega{c}({\mathcal{A}}\hat{\otimes}{\mathcal{A}})^{\ast})^{\ast}$ which has a $wk^{*}$-dense range. So there exists a net $(u_\beta^\alpha)_{\beta\in{\Theta}}$ in $\mathcal{A}\hat{\otimes}\mathcal{A}$ such that $wk^*\hbox{-}\lim\limits_{\beta}{u}_\beta^\alpha=m_{\alpha}$ in $(\sigma{wc}(\mathcal{A}\hat{\otimes}\mathcal{A})^{\ast})^*$. Thus for every $a\in{\mathcal{A}}$
 \begin{equation*}
 wk^*\hbox{-}\lim\limits_{\alpha}wk^*\hbox{-}\lim\limits_{\beta}(a\cdot u_\beta^\alpha-u_\beta^\alpha\cdot a)=wk^*\hbox{-}\lim\limits_{\alpha}(a\cdot m_{\alpha}-m_{\alpha}\cdot a)=0\quad\hbox{ in $(\sigma{wc}(\mathcal{A}\hat{\otimes}\mathcal{A})^{\ast})^*$}.
 \end{equation*}
Since $\pi_{\sigma{wc}}$ is $wk^*$-continuous and the multiplication in $\mathcal{A}$ is separately $wk^*$-continuous \cite[Exercise 4.4.1]{Runde:2002},	
\begin{equation}
 wk^*\hbox{-}\lim\limits_{\alpha}wk^*\hbox{-}\lim\limits_{\beta}a\pi_{\sigma{wc}}(u_\beta^\alpha)=wk^*\hbox{-}\lim\limits_{\alpha}a\pi_{\sigma{wc}}(m_{\alpha})=a\quad\hbox{ in $\mathcal{A}$}.
\end{equation}
Let $E=I\times\Theta^I$ be a directed set with product ordering defined by
\begin{equation*}
(\alpha,\beta)\leq_E(\alpha^\prime,\beta^\prime)\Leftrightarrow\alpha\leq_I\alpha^\prime,\beta\leq_{\Theta^I}\beta^\prime\qquad(\alpha,\alpha^\prime\in{I},\quad \beta,\beta^\prime\in{\Theta^I}),
\end{equation*}
where $\Theta^I$ is the set of all functions from $I$ into $\Theta$ and $\beta\leq_{\Theta^I}\beta^\prime$ means that $\beta(d)\leq_\Theta\beta^\prime(d)$ for every $d\in{I}$. Suppose that $\gamma=(\alpha,\beta_\alpha)$ and $n_\gamma=u_\beta^\alpha$. By iterated limit theorem \cite[Page 69]{kelley:75}, one can see that $wk^*\hbox{-}\lim\limits_{\gamma}a\cdot n_\gamma-n_\gamma\cdot a=0$ in $(\sigma{wc}(\mathcal{A}\hat{\otimes}\mathcal{A})^{\ast})^*$ and $wk^*\hbox{-}\lim\limits_{\gamma}a\pi_{\sigma{wc}}(n_\gamma)=a$ in $\mathcal{A}$. So $\mathcal{A}$ is pseudo-Connes amenable \cite[Definition 4.3]{Mahmoodi:14}.
\end{proof}
\begin{lemma}\label{l2.3}
	Let $\mathcal{A}$ be a commutative pseudo-Connes amenable dual Banach algebra. Then $\mathcal{A}$ is strong pseudo-Connes amenable.
\end{lemma}
\begin{proof}
Clear.	
\end{proof}
Recently a new notion of amenability for Banach algebra, Johnson pseudo-contractible and similarly a new notion of Connes amenability for dual Banach algebra, Johnson pseudo-Connes amenable introduced \cite{Sahami:2017}, \cite{Sha:17}. A dual Banach algebra $\mathcal{A}$ is called Johnson pseudo-Connes amenable, if there exists a not necessarily bounded net $(m_{\alpha})$ in $(\mathcal{A}\hat{\otimes}\mathcal{A})^{**}$ such that $\langle{T},a\cdot{m_{\alpha}}\rangle=\langle{T},{m_{\alpha}}\cdot{a}\rangle$ and $i^{\ast}_{\mathcal{A}_{\ast}}\pi_{\mathcal{A}}^{**}(m_{\alpha}){a}\rightarrow{a}$ for every $a\in{\mathcal{A}}$ and $T\in\sigma{wc}(\mathcal{A}\hat{\otimes}\mathcal{A})^{\ast}$, where $i_{\mathcal{A}_{\ast}}:\mathcal{A}_{\ast}\hookrightarrow\mathcal{A}^{\ast}$ is a canonical embedding \cite{Sha:17}. The notion of Johnson pseudo-Connes amenable is stronger than pseudo-Connes amenability \cite[Lemma 2.4]{Sha:17}. The authors showed that for a locally compact group $G$, $M(G)$ is Johnson pseudo-Connes amenable if and only if $G$ is amenable and Also for every non-empty set $I$, $\mathbb{M}_I(\mathbb{C})$ under this new notion is forced to have a finite index. For more details see \cite{Sha:17}.
\begin{Proposition}\label{p4}
Let $\mathcal{A}$ be a dual Banach algebra. If $\mathcal{A}$ is Johnson pseudo-Connes amenable, then $\mathcal{A}$ is strong pseudo-Connes amenable.
\end{Proposition}
\begin{proof}
Since $\mathcal{A}$ is Johnson pseudo-Connes amenable, there exists a net $(m_{\alpha})$ in $(\mathcal{A}\hat{\otimes}\mathcal{A})^{**}$ such that $\langle{T},a\cdot{m_{\alpha}}\rangle=\langle{T},{m_{\alpha}}\cdot{a}\rangle$ and $i^{\ast}_{\mathcal{A}_{\ast}}\pi_{\mathcal{A}}^{**}(m_{\alpha}){a}\rightarrow{a}$ for every $a\in{\mathcal{A}}$ and $T\in\sigma{wc}(\mathcal{A}\hat{\otimes}\mathcal{A})^{\ast}$. Let $\tilde{m}_\alpha=q(m_{\alpha})$, where $q:(\mathcal{A}\hat{\otimes}\mathcal{A})^{\ast\ast}\rightarrow(\sigma{wc}({\mathcal{A}}\hat{\otimes}{\mathcal{A}})^{\ast})^{\ast}$ is a quotient map as in Remark \ref{R1}. So $a\cdot\tilde{m}_\alpha=\tilde{m}_\alpha\cdot a$ for every $a\in{\mathcal{A}}$. Since $\pi_{\sigma{wc}}$ is an $\mathcal{A}$-bimodule homomorphism, $a\pi_{\sigma{wc}}(\tilde{m}_\alpha)=\pi_{\sigma{wc}}(\tilde{m}_\alpha)a$. For every $f\in{\mathcal{A}_{\ast}}$ we have
\begin{equation*}
\begin{split}
\langle f,\pi_{\sigma{wc}}q({m}_\alpha)\rangle&=\langle \pi^*\vert_{\mathcal{A}_{\ast}}(f),q({m}_\alpha)\rangle=\langle \pi^*\vert_{\mathcal{A}_{\ast}}(f),{m}_\alpha\rangle=\langle\pi^*(f),{m}_\alpha\rangle\\&=\langle f,\pi^{**}({m}_\alpha)\rangle=\langle i_{\mathcal{A}_{\ast}}(f),\pi^{**}({m}_\alpha)\rangle=\langle f,i^{\ast}_{\mathcal{A}_{\ast}}\pi^{**}({m}_\alpha)\rangle.
\end{split}
\end{equation*}
So 
\begin{equation}\label{ee2}
\pi_{\sigma{wc}}q=i^{\ast}_{\mathcal{A}_{\ast}}\pi
_{\mathcal{A}}^{**}.
\end{equation}
Thus $wk^*\hbox{-}\lim\limits_{\alpha}a\pi_{\sigma{wc}}(\tilde{m}_{\alpha})=wk^*\hbox{-}\lim\limits_{\alpha}\pi_{\sigma{wc}}(\tilde{m}_{\alpha})a=a$. Hence $\mathcal{A}$ is strong pseudo-Connes amenable.
\end{proof}
\begin{Proposition}\label{P2.5}
	Let $\mathcal{A}$ be a dual Banach algebra. If $\mathcal{A}$ is Connes amenable, then $\mathcal{A}$ is strong pseudo-Connes amenable.
\end{Proposition}	
\begin{proof}
	Let $\mathcal{A}$ be a Connes amenable Banach algebra. Then by \cite[Theorem 4.8]{Runde:2004}, there is an element ${M}\in{(\sigma{wc}(\mathcal{A}\hat{\otimes}\mathcal{A})^{\ast})^*}$ such that
	\begin{center}
		$a\cdot{M}={M}\cdot{a}\quad$ and $\quad\pi_{\sigma{wc}}({M})a=a\quad(a\in{\mathcal{A}})$. 
	\end{center}
Since $\pi_{\sigma{wc}}$ is an $\mathcal{A}$-bimodule homomorphism, $a\pi_{\sigma{wc}}({M})=\pi_{\sigma{wc}}({M})a$. So $\mathcal{A}$ is strong pseudo-Connes amenable.
\end{proof}
\begin{Remark}
	Let $\mathcal{A}$ be a strong pseudo-Connes amenable dual Banach algebra. In the Definition \ref{d1} if $(m_{\alpha})$  is considered as a bounded net, then by Banach-Alaoglu theorem, there is a $wk^*$-limit point for the net $(m_{\alpha})$. Let $M=wk^*\hbox{-}\lim\limits_{\alpha}m_{\alpha}$. One can see that $M$ is a $\sigma{wc}$-virtual diagonal for $\mathcal{A}$. So $\mathcal{A}$ is Connes amenable \cite[Theorem 4.8]{Runde:2004}. 
\end{Remark}
\begin{Proposition}
	Let $\mathcal{A}$ be a dual Banach algebra. If $\mathcal{A}$ is strong pseudo-amenable, then $\mathcal{A}$ is strong pseudo Connes-amenable.
\end{Proposition}
\begin{proof}
Since $\mathcal{A}$ is strong pseudo-amenable, there exists a net $(m_{\alpha})$ in $(\mathcal{A}\hat{\otimes}\mathcal{A})^{\ast\ast}$	such that for every $a\in{\mathcal{A}}$
\begin{equation}\label{ee3}
 a\cdot m_{\alpha}-m_{\alpha}\cdot a\rightarrow0,\quad 
 a\pi^{**}_{\mathcal{A}}(m_{\alpha})=\pi^{**}_{\mathcal{A}}(m_{\alpha})a\rightarrow a.
\end{equation}
 Let $\tilde{m}_\alpha=q(m_{\alpha})$, where $q:(\mathcal{A}\hat{\otimes}\mathcal{A})^{\ast\ast}\rightarrow(\sigma{wc}({\mathcal{A}}\hat{\otimes}{\mathcal{A}})^{\ast})^{\ast}$ is a quotient map as in Remark \ref{R1}. So $a\cdot\tilde{m}_\alpha-\tilde{m}_\alpha\cdot a\overset{wk^*}{\rightarrow}0$ for every $a\in{\mathcal{A}}$. By (\ref{ee2}) and (\ref{ee3}) we have
 \begin{equation*}
wk^*\hbox{-}\lim\limits_{\alpha}a\pi_{\sigma{wc}}(\tilde{m}_{\alpha})=wk^*\hbox{-}\lim\limits_{\alpha}\pi_{\sigma{wc}}(\tilde{m}_{\alpha})a=a.
 \end{equation*}
\end{proof}
 A dual Banach algebra $\mathcal{A}$ is called Connes biprojective if there exists a bounded $\mathcal{A}$-bimodule morphism $\rho:\mathcal{A}\longrightarrow(\sigma{wc}(\mathcal{A}\hat{\otimes}\mathcal{A})^{\ast})^{\ast}$ such that $\pi_{\sigma{wc}}\circ\rho=id_{\mathcal{A}}$ . Shirinkalam and second auther \cite{Shi:2016} showed that a dual Banach algebra $\mathcal{A}$ is Connes amenable if and only if $\mathcal{A}$ is Connes biprojective and has an identity.
\begin{Proposition}
	Let $\mathcal{A}$ be a dual Banach algebra with a central approximate identity. If $\mathcal{A}$ is Connes biprojective, then $\mathcal{A}$ is strong pseudo-Connes amenable.
\end{Proposition}
\begin{proof}
Let $\mathcal{A}$ be a Connes biprojective dual Banach algebra with a central approximate identity. Then $\mathcal{A}$ is Johnson pseudo-Connes amenable \cite[Proposition 2.5]{Sha:17}. So Proposition \ref{p4} implies that $\mathcal{A}$ is strong pseudo-Connes amenable. 
	\end{proof}
\begin{Proposition}\label{p9}
	Let $\mathcal{A}$ and $\mathcal{B}$ be  dual Banach algebras. Suppose that $\theta:\mathcal{A}\longrightarrow\mathcal{B}$ is a continuous epimorphism which is also $wk^*$-continuous. If $\mathcal{A}$ is strong pseudo-Connes amenable, then $\mathcal{B}$ is strong pseudo-Connes amenable.
\end{Proposition}
\begin{proof}
Since $\mathcal{A}$ is strong pseudo-Connes amenable, there exists a net $(m_{\alpha})$ in $(\sigma{wc}(\mathcal{A}\hat{\otimes}\mathcal{A})^{\ast})^{\ast}$	such that for every $a\in{\mathcal{A}}$
\begin{equation}\label{eqq}
a\cdot m_{\alpha}-m_{\alpha}\cdot a\overset{wk^*}{\longrightarrow}0,\quad
a\pi^{\mathcal{A}}_{\sigma{wc}}(m_{\alpha})=\pi^{\mathcal{A}}_{\sigma{wc}}(m_{\alpha})a\overset{wk^*}{\longrightarrow}a.
\end{equation}
Define $\theta\otimes\theta:\mathcal{A}\hat{\otimes}\mathcal{A}\rightarrow\mathcal{B}\hat{\otimes}\mathcal{B}$ by $\theta\otimes\theta(x\otimes y)=\theta(x)\otimes\theta(y)$, for every $x,y\in{\mathcal{A}}$. So $\theta\otimes\theta$ is a bounded linear map. For every $a\in{\mathcal{A}}$ and $u\in{\mathcal{A}\hat{\otimes}\mathcal{A}}$ we have
\begin{equation*}
\theta(a)\cdot(\theta\otimes\theta)(u)=(\theta\otimes\theta)(a\cdot u),\qquad(\theta\otimes\theta)(u)\cdot\theta(a)=(\theta\otimes\theta)(u\cdot a).
\end{equation*}
By \cite[Lemma 4.4]{Mahmoodi:14}, for every $a\in{\mathcal{A}}$ and $f\in{(\mathcal{B}\hat{\otimes}\mathcal{B})^*}$ we have
\begin{equation}\label{eq4}
a\cdot(\theta\otimes\theta)^{*}(f)=(\theta\otimes\theta)^{*}(\theta(a)\cdot f),\qquad (\theta\otimes\theta)^{*}(f)\cdot a=(\theta\otimes\theta)^{*}(f\cdot\theta(a)),
\end{equation}
and also
\begin{equation}
(\theta\otimes\theta)^{*}(\sigma wc(\mathcal{B}\hat{\otimes}\mathcal{B})^{*})\subseteq\sigma wc(\mathcal{A}\hat{\otimes}\mathcal{A})^{*}.
\end{equation}
Define the map
\begin{equation*}
\varPsi:=((\theta\otimes\theta)^{*}\vert_{\sigma wc(\mathcal{B}\hat{\otimes}\mathcal{B})^{*}})^*:(\sigma wc(\mathcal{A}\hat{\otimes}\mathcal{A})^{*})^*\rightarrow(\sigma wc(\mathcal{B}\hat{\otimes}\mathcal{B})^{*})^*.
\end{equation*}
Let $n_\alpha=\varPsi(m_\alpha)$. (\ref{eq4}) implies that for every $a\in{\mathcal{A}}$ and $T\in{\sigma wc(\mathcal{B}\hat{\otimes}\mathcal{B})^{*}}$
\begin{equation*}
\begin{split}
\langle T, \varPsi(a\cdot m_\alpha) \rangle&=\langle(\theta\otimes\theta)^{*}(T),a\cdot m_\alpha\rangle=\langle(\theta\otimes\theta)^{*}(T)\cdot a,m_\alpha\rangle=\langle(\theta\otimes\theta)^{*}(T\cdot\theta(a)),m_\alpha\rangle\\&=\langle(\theta\otimes\theta)^{*}\vert_{\sigma wc(\mathcal{B}\hat{\otimes}\mathcal{B})^{*}}(T\cdot\theta(a)),m_\alpha\rangle=\langle T\cdot\theta(a),\varPsi( m_\alpha)\rangle=\langle T,\theta(a)\cdot \varPsi(m_\alpha)\rangle,
\end{split}
\end{equation*}
and similarity for the right action for every $a\in{\mathcal{A}}$, $\varPsi(m_\alpha\cdot a)=\varPsi(m_\alpha)\cdot\theta(a)$. Since $\varPsi$ is $wk^*$-continuous,
\begin{equation*}
\lim\limits_{\alpha}\langle T, \theta(a)\cdot n_\alpha-n_\alpha\cdot\theta(a) \rangle=\lim\limits_{\alpha}\langle T,\varPsi(a\cdot m_\alpha-m_\alpha\cdot a)\rangle=0\quad(T\in{\sigma wc(\mathcal{B}\hat{\otimes}\mathcal{B})^{*}}).
\end{equation*}
So $\theta(a)\cdot n_\alpha-n_\alpha\cdot\theta(a)\overset{wk^*}{\rightarrow}0$.
By using the argument of \cite[Proposition 4.5 (ii)]{Mahmoodi:14} for every $u\in{\mathcal{A}\hat{\otimes}\mathcal{A}}$ we have $\pi^{\mathcal{B}}_{\sigma{wc}}\circ\theta\otimes\theta(u)=\theta\circ\pi^{\mathcal{A}}_{\sigma{wc}}$. Since the map $i:\mathcal{A}\hat{\otimes}\mathcal{A}\rightarrow(\sigma wc(\mathcal{A}\hat{\otimes}\mathcal{A})^*)^*$ has a $wk^*$-dense range, for every $\alpha$ there exists $(u_\beta^\alpha)$ in ${\mathcal{A}\hat{\otimes}\mathcal{A}}$ such that $wk^*\hbox{-}\lim\limits_{\alpha} u_\beta^\alpha=m_\alpha$ in $(\sigma wc(\mathcal{A}\hat{\otimes}\mathcal{A})^*)^*$. Since the maps $\pi^{\mathcal{B}}_{\sigma{wc}}$, $\varPsi$, $\theta$ and $\pi^{\mathcal{A}}_{\sigma{wc}}$ are $wk^*$-continuous,
\begin{equation*}
\begin{split}
\pi^{\mathcal{B}}_{\sigma{wc}}\circ\varPsi(m_\alpha)&=wk^*\hbox{-}\lim\limits_{\alpha}\pi^{\mathcal{B}}_{\sigma{wc}}\circ\varPsi(u_\beta^\alpha)=wk^*\hbox{-}\lim\limits_{\alpha}\pi^{\mathcal{B}}_{\sigma{wc}}\circ\theta\otimes\theta(u_\beta^\alpha)\\&=wk^*\hbox{-}\lim\limits_{\alpha}\theta\circ\pi^{\mathcal{A}}_{\sigma{wc}}(u_\beta^\alpha)=\theta\circ\pi^{\mathcal{A}}_{\sigma{wc}}(m_\alpha).
\end{split}
\end{equation*}
  Since $\theta$ is $wk^*$-continuous, (\ref{eqq}) implies that
\begin{equation*}
 \theta(a)\pi^{\mathcal{B}}_{\sigma{wc}}(\varPsi(m_{\alpha}))=\pi^{\mathcal{B}}_{\sigma{wc}}(\varPsi(m_{\alpha}))\theta(a)\overset{wk^*}{\longrightarrow}\theta(a).
\end{equation*}
So $\mathcal{B}$ is strong pseudo-Connes amenable.
\end{proof}
\begin{cor}
	Let $\mathcal{A}$ be a dual Banach algebra and let $I$ be a $wk^*$-closed ideal of $\mathcal{A}$. If $\mathcal{A}$ is strong pseudo-Connes amenable, then ${\mathcal{A}}/{I}$ is strong pseudo-Connes amenable.
\end{cor}
\begin{proof}
	Since the quotient map $q:\mathcal{A}\rightarrow{\mathcal{A}}/{I}$ is a $wk^*$-continuous map, by Proposition \ref{p9} the dual Banach algebra ${\mathcal{A}}/{I}$ is strong pseudo-Connes amenable. 
\end{proof}

\begin{lemma}\label{l2.12}
	Let $\mathcal{A}$ be a dual Banach algebra and $\varphi\in{\Delta_{wk^*}(\mathcal{A})}$. If $\mathcal{A}$ is strong pseudo-Connes amenable, then there is a net $(n_\alpha)$ in $\mathcal{A}$ such that 
	\begin{equation*}
	a n_\alpha-\varphi(a)n_\alpha\overset{wk^*}{\rightarrow}0,\qquad\varphi(n_\alpha)\rightarrow1\quad(a\in{\mathcal{A}}).
	\end{equation*}
\end{lemma}
	\begin{proof}
Since $\mathcal{A}$ is strong pseudo-Connes amenable, there is a net $(m_{\alpha})$ in $(\sigma{wc}(\mathcal{A}\hat{\otimes}\mathcal{A})^{\ast})^{\ast}$	such that for every $a\in{\mathcal{A}}$
\begin{equation*}
a\cdot m_{\alpha}-m_{\alpha}\cdot a\overset{wk^*}{\longrightarrow}0,\quad
a\pi_{\sigma{wc}}(m_{\alpha})=\pi_{\sigma{wc}}(m_{\alpha})a\overset{wk^*}{\longrightarrow}a.
\end{equation*}
Define $\theta:\mathcal{A}\hat{\otimes}\mathcal{A}\rightarrow\mathcal{A}$ by $\theta(a\otimes b)=\varphi(b)a$ for every $a,b\in{\mathcal{A}}$.  It is easy to see that 
\begin{equation}\label{e2.4}
a\cdot\theta^*(f)=\varphi(a)\theta^*(f),\quad \theta^*(f)\cdot a=\theta^*(f\cdot a)\quad(a\in{\mathcal{A}}, f\in{{\mathcal{A}}^{*}}),
\end{equation}
and also
\begin{equation}\label{e2.5}
\langle\theta(u),\varphi\rangle=\langle\pi(u),\varphi\rangle\qquad(u\in{\mathcal{A}\hat{\otimes}\mathcal{A}}).
\end{equation} 
Since $\varphi$ is $wk^*$-continuous, by (\ref{e2.4}) one can see that
\begin{equation*}
\theta^*(\mathcal{A}_*)\subseteq\sigma{wc}(\mathcal{A}\hat{\otimes}\mathcal{A})^{\ast}.
\end{equation*}
Define $\tau:=(\theta^*\vert_{\mathcal{A}_*})^*:(\sigma{wc}(\mathcal{A}\hat{\otimes}\mathcal{A})^{\ast})^\ast\rightarrow\mathcal{A}$. Let $n_{\alpha}=\tau(m_{\alpha})$ for every $\alpha$. For each $a\in{\mathcal{A}}$ and $f\in{{\mathcal{A}}_{*}}$ by (\ref{e2.4}), we have
\begin{equation}\label{e2.6}
\langle a n_{\alpha},f\rangle=\langle a\tau(m_{\alpha}),f\rangle=\langle\tau(m_{\alpha}),f\cdot a \rangle=\langle m_{\alpha},\theta^*(f\cdot a)\rangle=\langle m_{\alpha},\theta^*(f)\cdot a\rangle=\langle a\cdot m_{\alpha},\theta^*(f)\rangle,
\end{equation}
and also
\begin{equation}\label{e2.7}
\begin{split}
\langle\varphi(a)n_{\alpha},f\rangle&=\langle\tau(m_{\alpha}),\varphi(a)f\rangle=\langle m_{\alpha},\theta^*(\varphi(a)f)\rangle=\langle m_{\alpha},\varphi(a)\theta^*(f)\rangle\\&=\langle m_{\alpha},a\cdot\theta^*(f)\rangle=\langle m_{\alpha}\cdot a,\theta^*(f)\rangle.
\end{split}
\end{equation}
Since $\lim\limits_{\alpha}\langle a\cdot m_{\alpha}-m_{\alpha}\cdot a,\theta^*(f)\rangle=0$, by (\ref{e2.6}) and (\ref{e2.7}) we have $a n_\alpha-\varphi(a)n_\alpha\overset{wk^*}{\rightarrow}0$. By Goldstein's Theorem for every $u\in{(\sigma{wc}(\mathcal{A}\hat{\otimes}\mathcal{A})^{\ast})^{\ast}}$ there is a net $(x_{\alpha})$ in $\mathcal{A}\hat{\otimes}\mathcal{A}$ such that $wk^*\hbox{-}\lim\limits_{\alpha}x_\alpha=u$ in $(\sigma{wc}(\mathcal{A}\hat{\otimes}\mathcal{A})^{\ast})^{\ast}$. One can see that $\pi_{\sigma{wc}}(u)=wk^*\hbox{-}\lim\limits_{\alpha}\pi_{\sigma{wc}}(x_\alpha)=wk^*\hbox{-}\lim\limits_{\alpha}\pi(x_\alpha)$ and $\tau(u)=wk^*\hbox{-}\lim\limits_{\alpha}\tau(x_\alpha)=wk^*\hbox{-}\lim\limits_{\alpha}\theta(x_\alpha)$. So by (\ref{e2.5}) we have
\begin{equation*}
\langle\varphi,\pi_{\sigma{wc}}(u)\rangle=\lim\limits_{\alpha}\langle\varphi,\pi(x_\alpha)\rangle=\lim\limits_{\alpha}\langle\varphi,\theta(x_\alpha)\rangle=\langle\varphi,\tau(u)\rangle.
\end{equation*}
So for every $\alpha$
\begin{equation}\label{e2.8}
\langle\varphi,\pi_{\sigma{wc}}(m_{\alpha})\rangle=\langle\varphi,\tau(m_{\alpha})\rangle.
\end{equation}
Since $\varphi$ is $wk^*$-continuous, by \cite[Chapter V; Theorem 1.3]{Con:85} $\varphi\in{\mathcal{A}_*}$. So $\varphi(\pi_{\sigma{wc}}(m_{\alpha}))\varphi(a)\rightarrow\varphi(a)$. Thus $\varphi(\pi_{\sigma{wc}}(m_{\alpha}))\rightarrow1$ in $\mathbb{C}$. By (\ref{e2.8}), $\varphi(n_{\alpha})\rightarrow1$.
\end{proof}
\begin{Remark}\label{R12}
	Note that in the proof of Lemma \ref{l2.12} with the same conditions, if we define $\theta:\mathcal{A}\hat{\otimes}\mathcal{A}\rightarrow\mathcal{A}$ by $\theta(a\otimes b)=\varphi(a)b$ for every $a,b\in{\mathcal{A}}$, then there exists  a net $(n_\alpha)$ in $\mathcal{A}$ such that 
	\begin{equation*}
	n_\alpha a-\varphi(a)n_\alpha\overset{wk^*}{\rightarrow}0,\qquad\varphi(n_\alpha)\rightarrow1\quad(a\in{\mathcal{A}}).
	\end{equation*}
\end{Remark}
\section{Some applications}
Following \cite{eslam:04}, Let $\mathcal{A}$ be a Banach algebra, $I$ and $J$ be arbitrary nonempty index sets
and let $P$ be a $J\times I$ matrix over $\mathcal{A}$ such that $\Vert P\Vert_{\infty}=\sup\{\Vert P_{j,i}\Vert:j\in{J},i\in{I}\}\leq1$. The set 
of all $I\times J$ matrices over $\mathcal{A}$ with finite $\ell^1$-norm and product $XY=XPY$ is a Banach algebra, which is denoted by $LM(\mathcal{A},P)$ and it
is called the $\ell^1$-Munn $I\times J$ matrix algebra over $\mathcal{A}$ with sandwich matrix $P$.\\
 $M_{I}(\mathbb{C})$ the Banach algebra of $I\times I$-matrices over $\mathbb{C}$, with finite $\ell^1$-norm and matrix multiplication is a dual $\ell^1$-Munn algebra \cite{shoj:09}.
\begin{Theorem}
	Let $I$ be a non-empty set. Then $\mathbb{M}_{I}(\mathbb{C})$ is strong pseudo-Connes amenable if and only if $I$ is finite.
\end{Theorem}
\begin{proof}
	Let $\mathcal{A}=\mathbb{M}_{I}(\mathbb{C})$ be strong pseudo-Connes amenable. Then there is a net $(m_{\alpha})$ in $(\sigma{wc}(\mathcal{A}\hat{\otimes}\mathcal{A})^{\ast})^{\ast}$	such that for every $a\in{\mathcal{A}}$
	\begin{equation*}
	a\cdot m_{\alpha}-m_{\alpha}\cdot a\overset{wk^*}{\longrightarrow}0,\quad
	a\pi_{\sigma{wc}}(m_{\alpha})=\pi_{\sigma{wc}}(m_{\alpha})a\overset{wk^*}{\longrightarrow}a.
	\end{equation*}
	Let $a$ be a non-zero element of $\mathcal{A}$. Then there is a $\psi$ in $\mathcal{A}_{\ast}$ such that $a(\psi)\neq0$. By similar argument as in Proposition \ref{p2.2}, there is a bounded net $(n_\gamma)$ in $\mathcal{A}\hat{\otimes}\mathcal{A}$ such that,  $wk^*\hbox{-}\lim\limits_{\gamma}a\cdot n_\gamma-n_\gamma\cdot a=0$ in $(\sigma{wc}(\mathcal{A}\hat{\otimes}\mathcal{A})^{\ast})^*$ and $wk^*\hbox{-}\lim\limits_{\gamma}a\pi_{\sigma{wc}}(n_\gamma)=wk^*\hbox{-}\lim\limits_{\gamma}\pi_{\sigma{wc}}(n_\gamma)a=a$ in $\mathcal{A}$. So $wk^*\hbox{-}\lim\limits_{\gamma}a\pi_{\sigma{wc}}(n_\gamma)-\pi_{\sigma{wc}}(n_\gamma)a=0$.
	For every $f\in{\mathcal{A}_{*}}$, we have
	\begin{equation*}
	\langle f,\pi_{\sigma{wc}}(n_\gamma)\rangle=\langle\pi^*\vert_{\mathcal{A}_{\ast}}(f),(n_\gamma)\rangle=\langle\pi^*(f),n_\gamma\rangle=\langle f,\pi(n_\gamma)\rangle.
	\end{equation*}
	So $	wk^*\hbox{-}\lim\limits_{\gamma}a\pi(n_\gamma)-\pi(n_\gamma)a=0$ and $wk^*\hbox{-}\lim\limits_{\gamma}a\pi(n_\gamma)=wk^*\hbox{-}\lim\limits_{\gamma}\pi(n_\gamma)a=a$ in $\mathcal{A}$.  
	Let $y_{\gamma}=\pi(n_\gamma)$. Then $(y_{\gamma})$ is a bounded net in $\mathcal{A}$ which satisfies
	\begin{equation}\label{eq31}
	wk^*\hbox{-}\lim\limits_{\gamma}ay_{\gamma}-y_{\gamma}a=0\quad\hbox{and}\quad	wk^*\hbox{-}\lim\limits_{\gamma}ay_{\gamma}=wk^*\hbox{-}\lim\limits_{\gamma}y_{\gamma}a=a\quad(a\in{\mathcal{A}}). 
	\end{equation}	
	By similar argument as in \cite[Theorem 3.2]{Sha:17},
	Suppose that $y_{\gamma}=[y_{\gamma}^{i,j}]$, where $y_{\gamma}^{i,j}\in{\mathbb{C}}$ for every $i,j$. Fixed $i_{0}\in{I}$, for every $j\in{I}$ we have 
	\begin{equation*}
	\varepsilon_{i_0,j}y_{\gamma}-y_{\gamma}\varepsilon_{i_0,j}=\sum\limits_{\underset{i\neq{j}}{i\in{I}}}y_{\gamma}^{j,i}\varepsilon_{i_0,i}+(y_{\gamma}^{j,j}-y_{\gamma}^{i_0,i_0})\varepsilon_{i_0,j}-\sum\limits_{\underset{i\neq{i_0}}{i\in{I}}}y_{\gamma}^{i,i_0}\varepsilon_{i,j},
	\end{equation*}
	where $\varepsilon_{i,j}$ is a matrix belongs to $\mathbb{M}_{I}(\mathbb{C})$ which $(i,j)$-th entry is $1$ and others are zero. Let $$X_{\gamma}=[X^{i,j}_\gamma]=\varepsilon_{i_0,j}y_{\gamma}-y_{\gamma}\varepsilon_{i_0,j}.$$ 
	(\ref{eq31}) implies that
	$wk^*\hbox{-}\lim\limits_{\gamma}X_{\gamma}=0$. Consider $\varepsilon_{i_0,j}$, $\varepsilon_{i_0,i}$ as elements in $\mathcal{A}_*$, whenever $i\neq j$ in ${I}$. So
	\begin{equation}\label{eq3.3}
	\lim\limits_{\gamma}y_{\gamma}^{j,j}-y_{\gamma}^{i_0,i_0}=\lim\limits_{\gamma}X^{i_0,j}_\gamma=\lim\limits_{\gamma}\langle\varepsilon_{i_0,j},X_\gamma\rangle=0,
	\end{equation}
	and
	\begin{equation}\label{eq3.4}
	\lim\limits_{\gamma}y_{\gamma}^{j,i}=\lim\limits_{\gamma}X^{i_0,i}_\gamma=\lim\limits_{\gamma}\langle\varepsilon_{i_0,i},X_\gamma\rangle=0.
	\end{equation}
	Since $\Vert{y}_{\gamma}\Vert\leq\Vert{m_{\alpha}}\Vert$, $(y_{\gamma}^{i_0,i_0})$ is a bounded net in $\mathbb{C}$. So it has a convergent subnet $(y_{\gamma_{k}}^{i_0,i_0})$ in $\mathbb{C}$. We may assume  that $\lim\limits_{\gamma_k}y_{\gamma_{k}}^{i_0,i_0}=l$. By (\ref{eq3.3}) $\lim\limits_{\gamma_k}y_{\gamma_{k}}^{i_0,i_0}-y_{\gamma_{k}}^{j,j}=0$. It follows that $\lim\limits_{\gamma_{k}}y_{\gamma_{k}}^{j,j}=l$ for every $j\in{I}$. If $l=0$, then by (\ref{eq3.4}) for every $i,j\in{I}$, $\lim\limits_{\gamma_{k}}y_{\gamma_{k}}^{i,j}=0$ in $\mathbb{C}$. So $wk$-$\lim\limits_{\gamma_{k}}y_{\gamma_{k}}^{i,j}=0$, where $i,j\in{I}$. Applying \cite[Theorem 4.3]{scha:71}, $wk$-$\lim\limits_{\gamma_{k}}y_{\gamma_{k}}=0$ in $\mathcal{A}$. It follows that $\lim\limits_{\gamma_{k}}\langle{y}_{\gamma_k},a\cdot{\psi}\rangle=0$. On the other hand (\ref{eq31})
	\begin{equation*}
	\lim\limits_{\gamma_{k}}\langle{y}_{\gamma_k},a\cdot{\psi}\rangle=\lim\limits_{\gamma_{k}}\langle{a}\cdot{\psi},{y}_{\gamma_k}\rangle=\lim\limits_{\gamma_{k}}\langle{\psi},{y}_{\gamma_k} a\rangle=\langle\psi,a\rangle\neq0,
	\end{equation*}
	which is a contradiction. So $wk$-$\lim\limits_{\gamma_{k}}y_{\gamma_{k}}^{j,j}=l\neq0$ for every $j\in{I}$. Using (\ref{eq3.4}) we have $wk$-$\lim\limits_{\gamma_k}y_{\gamma_k}^{j,i}=0$ whenever $j\neq{i}$ in $I$. Applying \cite[Theorem 4.3]{scha:71} again, $wk$-$\lim\limits_{\gamma_{k}}y_{\gamma_{k}}=y_{0}$, where $y_0$ is a matrix with $l$ in the diagonal position and $0$ elsewhere. Thus $y_0\in{\overline{Conv(y_{\gamma_k})}}^{wk}={\overline{Conv(y_{\gamma_k})}}^{\Vert\cdot\Vert}$. So $y_0\in{\mathcal{A}}$. But 
	\begin{equation*}
	\infty=\sum\limits_{j\in{I}}\vert l\vert=\sum\limits_{j\in{I}}\vert y_0^{j,j}\vert=\Vert y_0\Vert<\infty,
	\end{equation*}
	which is a contradiction. So $I$ must be finite.\\
	Converesely, if $I$ is finite, then $\mathbb{M}_{I}(\mathbb{C})$ is Connes amenable \cite[Theorem 3.7]{Mah:2016}. So by Proposition \ref{P2.5}, $\mathbb{M}_{I}(\mathbb{C})$ is 
	strong pseudo-Connes amenable. 
\end{proof}
 Let $\mathcal{A}$ be a dual Banach algebra and let $I$ be a totally ordered set. Then the set of all $I\times{I}$-upper triangular matrices
with the usual matrix operations and the norm $\parallel[a_{i,j}]_{{i,j}\in{I}}\parallel=\sum\limits _{i,j\in{I}}\parallel{a}_{i,j}\parallel<\infty$, becomes a Banach algebra and it is denoted by
$$UP(I,\mathcal{A})=\set{\left[
	\begin{array}{rr} a_{i,j}  \end{array} \right]_{i,j\in I};  a_{i,j}\in A {\hbox{ and\, $a_{i,j}=0$ \,for every } }i>j }.$$
Authors showed that $\A$ is a dual Banach algebra \cite[Theorem 3.1]{Shar:18}.
\begin{Theorem}\label{T7}
	Let $\mathcal{A}$ be a dual Banach algebra with $\varphi\in{\Delta_{wk^*}(\mathcal{A})}$ and let $I$ be a totally ordered set with smallest element. Then $\A$ is strong pseudo-Connes amenable if and only if $\mathcal{A}$ is strong pseudo-Connes amenable and $\vert I\vert=1$.
\end{Theorem}
\begin{proof}
Let $\A$ be strong pseudo-Connes amenable. Assume that $i_0$ be a smallest element and  $\varphi\in{\Delta_{{wk}^{\ast}}{(\mathcal{A})}}$. We define a map $\psi:\A\longrightarrow\mathbb{C}$ by          $\left[ a_{i,j}\right] _{i,j\in{I}}\longmapsto\varphi{(a_{i_{0},i_{0}})}$ for every $\left[ a_{i,j}\right] _{i,j\in{I}}\in{\A}$.
Since $\varphi$ is $wk^\ast$-continuous, $\psi\in{\Delta_{{wk}^{\ast}}{(\A)}}$. By Lemma \ref{l2.12} and Remark \ref{R12}, there exists a net $(n_\alpha)$ in $\A$ such that 
\begin{equation*}
n_\alpha a-\psi(a)n_\alpha\overset{wk^*}{\rightarrow}0,\qquad\psi(n_\alpha)\rightarrow1\quad(a\in{\A}).
\end{equation*}
 Using the argument of \cite[Theorem 3.1]{Sah:18}, Let
 \begin{equation*}
 J=\big\{\left[ a_{i,j}\right] _{i,j\in{I}}\in{\A}\mid a_{i,j}=0\quad\forall i\neq i_0\big\}.
 \end{equation*}
 One can see that $J$ is a $wk^*$-closed ideal in $\A$ and $\psi\vert_J\neq0$. Consider $j_0\in{J}$ such that $\psi(j_0)=1$. Let $m_\alpha=j_0 n_\alpha$. since the multiplication in $UP(I,\mathcal{A})$ is separately $wk^\ast$-continuous \cite[Exercise 4.4.1]{Runde:2002}, we may assume that $(m_\alpha)$ is a net in $J$ such that 
 \begin{equation}\label{eq34}
m_\alpha a-\psi(a)m_\alpha\overset{wk^*}{\rightarrow}0,\qquad\psi(m_\alpha)\rightarrow1\quad(a\in{J}).
 \end{equation}
Suppose that $\lvert{I}\rvert>1$ and $m_\alpha$ has a form $\left(\begin{array}{ccc}a^\alpha_{i_0,i_0}&a^\alpha_{i_0,i}&\cdots\\
	0&0&\cdots\\
	\colon&\cdots&\colon
\end{array}
\right)$, for some nets $(a^\alpha_{i_0,i_0})$, $(a^\alpha_{i_0,i})$ in $\mathcal{A}$. So $\varphi(a^\alpha_{i_0,i_0})=\psi(m_\alpha)\rightarrow1$. Consider $x\in{\mathcal{A}}$ such that $\varphi(x)=1$. Let $a=\left(
\begin{array}{cccc}
0&x&0&\cdots\\
0&0&0&\cdots\\
\colon&\colon&\cdots&\colon
\end{array}
\right)$ in $J$. Since $\psi(a)=0$, (\ref{eq34}) implies that $m_\alpha a\overset{wk^*}{\rightarrow}0$. By a simple computation $a^\alpha_{i_0,i_0}x\overset{wk^*}{\rightarrow}0$. Since $\varphi$ is $wk^*$-continuous,
\begin{equation*}
\varphi(a^\alpha_{i_0,i_0})=\varphi(a^\alpha_{i_0,i_0})\varphi(x)=\varphi(a^\alpha_{i_0,i_0}x){\rightarrow}0,
\end{equation*}
  which is a contradiction. So $\lvert{I}\rvert=1$ and $\mathcal{A}=\A$ is strong pseudo-Connes amenable.

Converse is clear.	
\end{proof}
\section{Examples}
\begin{Example}
		Consider the Banach algebra $\ell^1$ of all sequences $a=(a_n)$ of complex numbers with 
	\begin{equation*}
	\Vert{a}\Vert=\sum\limits_{n=1}^\infty\vert a_n\vert<\infty,
	\end{equation*}
	and the following product
	\begin{equation*}
	(a\ast b)(n)=\left\{
	\begin{array}{ll}
	a(1)b(1)& \hbox{if}\quad n=1\\
	a(1)b(n)+b(1)a(n)+a(n)b(n)&\hbox{if}\quad n>1
	\end{array}
	\right.
	\end{equation*}
	for every $a,b\in{\ell^1}$. It is easy to see that $\Delta(\ell^1)=\{\varphi_1\}\cup\{\varphi_1+\varphi_n:n\geq2\}$, where $\varphi_n(a)=a(n)$ for every $a\in{\ell^1}$. 	By similar argument as in \cite[Example 4.1]{Sha:17},  $(\ell^1,*)$ is a dual Banach algebra with respect to $c_0$.  We show that $c_0$ is an $\ell^1$-module with dual actions. In fact for every $a\in{\ell^1}$ and $\lambda\in{c_0}$ we have
	\begin{equation*}
	a\cdot\lambda(n)=\left\{
	\begin{array}{ll}
	\sum\limits_{k=1}^\infty a(k)\lambda(k)& \hbox{if}\quad n=1\\
	(a(1)+a(n))\lambda(n)&\hbox{if}\quad n>1.
	\end{array}
	\right.
	\end{equation*}
	Since $\lambda$ vanishes at infinity and $\underset{n}{\sup}\vert a(n)\vert<\infty$, one can see that $a\cdot\lambda$ vanishes at infinity and  similarity for the right action. So $c_0$ is a closed $\ell^1$-submodule of $\ell^\infty$. We claim that $\ell^1$ is not strong pseudo-Connes amenable. Suppose conversely that $\ell^1$ is strong pseudo-Connes amenable. Since $\varphi_1$ is $wk^*$-continuous, by Lemma \ref{l2.12} there is a bounded net $(m_{\alpha})$ in $\ell^1$ that satisfies
\begin{equation}\label{eq41}
a*{m_{\alpha}}-\varphi_1(a){m_{\alpha}}{\overset{wk^\ast}{\longrightarrow}}0\quad\hbox{and}\quad\varphi_1(m_{\alpha})\longrightarrow1\qquad({a}\in{\ell^1}).
\end{equation}
Using the argument of \cite[Example 4.1]{Sha:17} again. Choose $a=\delta_{n}$ in $\ell^1$, where $n\geq2$. So $\varphi_1(\delta_{n})=0$. (\ref{eq41}) implies that  $\delta_{n}*{m_{\alpha}}{\overset{wk^\ast}{\longrightarrow}}0$ in $\ell^1$. One can see that  $\delta_{n}\ast m_\alpha=({m}_\alpha(1)+{m}_\alpha(n))\delta_{n}$. Consider $\delta_{n}$ as an element in $c_0$, where $n\geq2$. So
\begin{equation*}
\lim\limits_{\alpha}\langle\delta_{n},\delta_{n}\ast m_\alpha\rangle=\lim\limits_{\alpha}{m}_\alpha(1)+{m}_\alpha(n)=0.
\end{equation*}
Since $\lim\limits_{\alpha}m_{\alpha}(1)=1$ and $\lim\limits_{\alpha}m_{\alpha}(n)=-1$ for every $n\geq2$, we have $\underset{\alpha}{\sup}\Vert m_\alpha\Vert=\infty$, which  contradicts  the boundedness of the net $(m_{\alpha})$.
\end{Example}
\begin{Example}
	Let $S$ be the set  of natural numbers $\mathbb{N}$ with the binary operation $(m,n)\longmapsto\max\{m,n\}$, where $m$ and $n$ are in $\mathbb{N}$. Then $S$ is a weakly cancellative semigroup, that is, for every $s,t\in{S}$ the set $\{x\in{S}:sx=t\}$ is finite. So $\ell^1(S)$ is a dual Banach algebra with predual $c_{0}(S)$ \cite[Theorem 4.6]{Dales:2010}. By \cite[Example 6.1]{Mahmoodi:14}, $\ell^1(S)$ is pseudo-Connes amenable. Since $\ell^1(S)$ is commutative, by Lemma \ref{l2.3} it is strong pseudo-Connes amenable. But $\ell^1(S)$ is not Connes amenable \cite[Theorem 5.13]{Daws:2006}.
\end{Example}	
\section{A note on ultra central approximate identity for Banach algebras}
Jabbari {\it et al.} have introduced the notion of $\varphi$-inner amenability \cite{Jab:11}. For a given Banach algebra $\mathcal{A}$, Let $\mathcal{A}_\varphi=\{a\in{\mathcal{A}}:\varphi(a)=1\}$, where $\varphi$ is a linear multiplication functional on $\mathcal{A}$. A Banach algebra $\mathcal{A}$ is called  $\varphi$-inner amenable if there exists
a bounded linear functional $m$ on $\mathcal{A}^*$ satisfying $m(\varphi)=1$ and $m(f\cdot a) =
m(a\cdot f)$ for every $f\in{\mathcal{A}^*}$ and for every $a\in{\mathcal{A}_\varphi}$. This notion is equivalent with the existence of a bounded
net $(a_{\alpha})$ in $\mathcal{A}$ such that $aa_{\alpha}-a_{\alpha}a\rightarrow0$ and $\varphi(a_{\alpha}) = 1$, for every $a\in{\mathcal{A}}$ \cite[Theorem 2.1]{Jab:11}.
\begin{Theorem}\label{T5.1}
	Let $\mathcal{A}$ be a Banach algebra and $\varphi\in\Delta(\mathcal{A})$. If $\mathcal{A}^{**}$ has an ultra central approximate identity, then $\mathcal{A}$ is $\varphi$-inner amenable.
\end{Theorem}
\begin{proof}
	Suppose that $\mathcal{A}^{**}$ has an ultra central approximate identity. Then there exists a net $(e_{\alpha})$ in $\mathcal{A}^{****}$ such that
	\begin{equation}\label{e5.1}
	ae_{\alpha}=e_{\alpha}a,\quad ae_{\alpha}\rightarrow a,\qquad (a\in \mathcal{A}^{**}).
	\end{equation}
	  We denote $\tilde{\varphi}$ for unique extension of $\varphi$ to $\mathcal{A}^{**}$ is defined by $\tilde{\varphi}(\lambda)=\lambda(\varphi)$ for every $\lambda\in{\mathcal{A}^{**}}$ and also $\tilde{\tilde{\varphi}}$ is denoted for the unique extension of $\tilde{\varphi} $ to $\mathcal{A}^{****}$ is defined in the same way. Clearly $\tilde{\tilde{\varphi}}\in\Delta(\mathcal{A}^{****})$ and  $\tilde{\varphi} \in \Delta(\mathcal{A}^{**})$. So $\tilde{\tilde{\varphi}}( e_{\alpha})\rightarrow 1$ for every $a\in \mathcal{A}^{**}$. Thus for a sufficient large $\alpha$, $\tilde{\tilde{\varphi}}(e_{\alpha})$ stays away from zero. Replacing $\frac{e_{\alpha}}{\tilde{\tilde{\varphi}}( e_{\alpha})}$ with $e_{\alpha}$, we may assume that $\tilde{\tilde{\varphi}}( e_{\alpha})=1$. By Goldstein's Theorem for every $\alpha$, there exists a net $(m^{\alpha}_{\beta})_{\beta}$ in $\mathcal{A}^{**}$ such that $ wk^*\hbox{-}\lim\limits_{\beta}m^{\alpha}_{\beta}=e_{\alpha}$ in $\mathcal{A}^{****}$ and $\Vert m^{\alpha}_{\beta}\Vert\leq \Vert e_{\alpha}\Vert$.
So by (\ref{e5.1}), $wk^*\hbox{-}\lim\limits_{\beta}am^{\alpha}_{\beta}-m^{\alpha}_{\beta}a=0$ in $\mathcal{A}^{****}$ for every $a\in{\mathcal{A}}$. Thus $wk\hbox{-}\lim\limits_{\beta}am^{\alpha}_{\beta}-m^{\alpha}_{\beta}a=0$ in $\mathcal{A}^{**}$ for every $a\in{\mathcal{A}}$. Let $F=\{a_{1},a_{2},...,a_{n}\}$ be an arbitrary finite subset of $\mathcal{A}$. Set $$V=\{(a_{1}b-ba_{1},a_{2}b-ba_{2},...,a_{n}b-ba_{n},\tilde{\varphi}(b)-1)|b\in \mathcal{A}^{**}, \Vert b\Vert\leq\Vert e_{\alpha}\Vert\}\leq \prod^{n}_{i=1}\mathcal{A}_{i}\oplus_{1} \mathbb{C},$$ where for every $i$, $\mathcal{A}_{i}=\mathcal{A}^{**}$. It is easy to see that $V$ is a convex set and $(0,0,...,0)$ is a $wk$-limit point of $V$. Since $(0,0,...,0)\in {\overline{V}}^{w}={\overline{V}}^{||\cdot||}$, there exists a bounded net $(n^{\alpha}_{\beta})_{\beta}$ in $\mathcal{A}^{**}$ such that $$an^{\alpha}_{\beta}-n^{\alpha}_{\beta}a\rightarrow 0,\quad \tilde{\varphi}(n^{\alpha}_{\beta})-1\rightarrow 0,\quad (a\in{\mathcal{A}}).$$ By Banach-Alaghlou's theorem, $(n^{\alpha}_{\beta})_{\beta}$ has a $wk^{*}$-limit point say $N\in \mathcal{A}^{**}$. One can see that $$aN=Na,\quad \tilde{\varphi}(N)=1,\quad (a\in \mathcal{A}).$$
	It follows that $\mathcal{A}$ is $\varphi$-inner amenable.
\end{proof}
Using the similar argument as in the proof of the theorem \ref{T5.1}, we have the following corollary:
\begin{cor}\label{C5.2}
	Let $\mathcal{A}$ be a Banach algebra and $\varphi\in\Delta(\mathcal{A})$. If $\mathcal{A}$ has an ultra central approximate identity, then $\mathcal{A}$ is $\varphi$-inner amenable.
\end{cor}
\begin{Example}
	Let $\mathcal{A}=\left(\begin{array}{cccc}\mathbb{C}&\mathbb{C}&\mathbb{C}&\cdots\\
	0&0&0&\cdots\\
	\colon&\cdots&\colon&\cdots
	\end{array}
	\right)_{\mathbb{N}\times\mathbb{N}}$. With the finite $\ell^{1}$-norm and matrix multiplication, $\mathcal{A}$ becomes a Banach algebra. We claim that $\mathcal{A}^{**}$ doesn't have an ultra central approximate identity. We go toward a contradiction and suppose that $\mathcal{A}^{**}$ has an ultra central approximate identity. Define $\varphi:\mathcal{A}\rightarrow \mathbb{C}$ by $\varphi((a_{i,j})_{i,j\in\mathbb{N}})=a_{11}$ for every $(a_{i,j})_{i,j\in\mathbb{N}}\in{\mathcal{A}}$. It is easy to see that $\varphi\in\Delta(\mathcal{A})$. Thus by Theorem \ref{T5.1}, $\mathcal{A}$ is $\varphi$-inner amenable. By \cite[Theorem 2.1]{Jab:11} there exists a bounded net  $n_{\alpha}=\left(\begin{array}{cccc}n^{1}_{\alpha}&n^{2}_{\alpha}&n^{3}_{\alpha}&\cdots\\
	0&0&0&\cdots\\
	\colon&\cdots&\colon&\cdots
	\end{array}
	\right)_{\mathbb{N}\times\mathbb{N}}$ in $\mathcal{A}$ such that
	\begin{equation}\label{eq5.2}
	an_{\alpha}-n_{\alpha}a\rightarrow 0,\quad \varphi(n_{\alpha})=n^{1}_{\alpha}\rightarrow 1,
	\end{equation}
	 for all $a\in \mathcal{A}$. Put $a_{1}=\left(\begin{array}{cccc}0&1&0&\cdots\\
	0&0&0&\cdots\\
	\colon&\cdots&\colon&\cdots
	\end{array}
	\right)_{\mathbb{N}\times\mathbb{N}}$ instead of $a$ in (\ref{eq5.2}), we have $n^{1}_{\alpha}\rightarrow 0$ which is a contradiction. So $\mathcal{A}^{**}$ doesn't have an ultra central approximate identity.
\end{Example}
\begin{Example}
	Let $S$ be a left zero semigroup with $|S|\geq 2$, that is, a semigroup with $st=s,$ for all $s,t\in S.$ We claim that the semigroup algebra $\ell^{1}(S)$ doesn't have an ultra central approximate identity. Suppose conversely that $\ell^{1}(S)$ has an ultra central approximate identity. It is easy to see that $fg=\varphi_{S}(g)f,$ where $\varphi_{S}$ is denoted for the augmentation character on $\ell^{1}(S)$ \cite[Page 54]{Dales:2010}. Applying Corollary \ref{C5.2}, follows that $\ell^{1}(S)$ is $\varphi_{S}$-inner amenable. Then by \cite[Theorem 2.1]{Jab:11} there exists a bounded net  $(n_{\alpha})$ in $\ell^{1}(S)$ such that
	\begin{equation}\label{eq5.3}
fn_{\alpha}-n_{\alpha}f\rightarrow 0,\quad \phi_{S}(n_{\alpha})\rightarrow 1,
	\end{equation}  
	 for all $f\in \ell^{1}(S).$ Since $|S|\geq 2$, set $f=\delta_{s_{1}}$ and $f=\delta_{s_{2}}$ and put $f$ in (\ref{eq5.3}) we have $n_{\alpha}\rightarrow \delta_{s_{1}}$ and $n_{\alpha}\rightarrow \delta_{s_{2}}$ which is a contradiction. Thus $\ell^{1}(S)$ doesn't have an ultra central approximate identity.
\end{Example}
We present some notions of semigroup theory. The semigroup $S$ is called left cancellative, if for every $a,b,c$ in $S$, 
\begin{equation*}
ca=cb\Rightarrow a=b.
\end{equation*}
The semigroup $S$ is called regular semigroup, if for every $s\in{S}$ there exists $s^*\in{S}$ such that $s=ss^*s$ and $s^*=s^*ss^*$ \cite{How:76}.
\begin{Theorem}
	Let $S$ be a left cancellative regular semigroup. Then ${\ell^{1}(S)}^{**}$ has an ultra central approximat identity if and only if $S$ is a group.
\end{Theorem}
\begin{proof}
	Suppose that  ${\ell^{1}(S)}^{**}$ has an ultra central approximate identity. Then there exists a net $(e_{\alpha})$ in ${\ell^{1}(S)}^{****}$ such that $ae_{\alpha}=e_{\alpha}a$ and $e_{\alpha}a\rightarrow a, $ for every $a\in {\ell^{1}(S)}^{**}$. By Goldstein's Theorem there exists a net $(x^{\alpha}_{\beta})$ in ${\ell^{1}(S)}^{**}$  such that $wk^*\hbox{-}\lim\limits_{\beta}x^{\alpha}_{\beta}=e_{\alpha}$ in ${\ell^{1}(S)}^{****}$. Thus $wk^*\hbox{-}\lim\limits_{\beta}ax^{\alpha}_{\beta}-x^{\alpha}_{\beta}a=0$ in ${\ell^{1}(S)}^{****}$ for every  $a\in \ell^{1}(S)$. So 
	\begin{equation*}
	wk^*\hbox{-}\lim\limits_{\alpha}wk^*\hbox{-}\lim\limits_{\beta}ax^{\alpha}_{\beta}-x^{\alpha}_{\beta}a=0,\quad wk^*\hbox{-}\lim\limits_{\alpha}wk^*\hbox{-}\lim\limits_{\beta}ax^{\alpha}_{\beta}=a,\quad(a\in \ell^{1}(S)),
	\end{equation*}
	 in ${\ell^{1}(S)}^{****}$. Applying iterated limit theorem \cite[Page 69]{kelley:75}, we have  a net $(E_{\gamma})$ in ${\ell^{1}(S)}^{**}$ such that
	 \begin{equation*}
		wk^*\hbox{-}\lim\limits_{\gamma}aE_{\gamma}-E_{\gamma}a=0,\quad 	wk^*\hbox{-}\lim\limits_{\gamma}aE_{\gamma}=a,\quad(a\in \ell^{1}(S)),
	 \end{equation*} 
in ${\ell^{1}(S)}^{****}$. It is easy to see that
\begin{equation*}
wk\hbox{-}\lim\limits_{\gamma}aE_{\gamma}-E_{\gamma}a=0,\quad wk\hbox{-}\lim\limits_{\gamma}aE_{\gamma}=a,\quad(a\in \ell^{1}(S)),
\end{equation*}
 in ${\ell^{1}(S)}^{**}$. Let $F=\{a_{1},a_{2},...,a_{n}\}$ be a finite subset of $\ell^{1}(S)$. Set
 \begin{equation*}
V=\{(a_{1}b-ba_{1},a_{2}b-ba_{2},...,a_{n}b-ba_{n},a_{1}b-a_{1},a_{2}b-a_{2},...,a_{n}b-a_{n})|b\in{\ell^{1}(S)}^{**}\}.
 \end{equation*}
  It is easy to see that $V$ is a convex subset of  $\prod\limits_{i=1}^{2n}\mathcal{A}_{i}$, where for every $i$, $\mathcal{A}_{i}={\ell^{1}(S)}^{**}$ and $(0,0,...,0)\in\overline{V}^{w}=\overline{V}^{||\cdot||}$. So there exists a net (say again $(E_{\gamma})$) in  ${\ell^{1}(S)}^{**} $  which
  \begin{equation*}
 \lim\limits_{\gamma}aE_{\gamma}-E_{\gamma}a=0,\quad \lim\limits_{\gamma}aE_{\gamma}=a,
  \end{equation*}
 for every $a\in \ell^{1}(S)$. Using a similar arguments we may suppose that $(E_{\gamma})$ is a net in $\ell^{1}(S)$ such that 
  \begin{equation*}
\lim\limits_{\gamma}aE_{\gamma}-E_{\gamma}a=0,\quad \lim\limits_{\gamma}aE_{\gamma}=a,
\end{equation*} 
for every $a\in \ell^{1}(S)$. So $ \ell^{1}(S)$ has an approximate identity. It follows that $\delta_{s}E_{\gamma}\rightarrow \delta_{s}$. Then there exists $t\in  S$ such that $st=s$. Suppose conversely that for every $t\in  S$, $st\neq s$. Assume that $E_{\gamma}$ has a form $\sum\limits_{t\in{S}} b^\gamma_t \delta_{t}$, where $b^\gamma_t \in{\mathbb{C}}$ for every $t\in{S}$. Let $\varepsilon=\frac{1}{2}$. Since $\delta_{s}E_{\gamma}\rightarrow \delta_{s}$, there exists $\gamma_0$ such that $\Vert\sum\limits_{t\in{S}} b^{\gamma_0}_t \delta_{st}-\delta_s\Vert_1\leq\frac{1}{2}$. Since $st\neq s$, $\sum\limits_{t\in{S}}\vert b^{\gamma_0}_t\vert+1\leq\frac{1}{2}$. So $\sum\limits_{t\in{S}}\vert b^{\gamma_0}_t\vert\leq-\frac{1}{2}$, which is a contradiction. Thus for every $s_{0}\in S$ we have $sts_{0}=ss_{0}$. The left cancellativity of $S$ implies that $ts_{0}=s_{0}$. Hence $t$ is a left identity for $S$. So 
\begin{equation}\label{e5.4}
\delta_{t}=\lim\limits_{\gamma}\delta_{t}E_{\gamma}=\lim\limits_{\gamma}E_{\gamma}.
\end{equation}
  For every $s\in S$ by (\ref{e5.4}), $\lim\limits_{\gamma}\delta_{s}E_{\gamma}=\delta_{st}$. On the other hand, $(E_{\gamma})$ is an approximate identity for $ \ell^{1}(S)$. So $\lim\limits_{\gamma}\delta_{s}E_{\gamma}=\delta_{s}$. It follows that  $s=st$, so $t$ is the identity of $S$. Since $S$ is a regular, for every $s\in S$ there exists $s^{*}\in S$ such that $$ss^{*}s=s=st,\quad s^{*}ss^{*}=s^{*}=s^{*}t.$$ Again the left cancellativity of $S$ implies that $s^{*}s=ss^{*}=t.$ It deduces that $S$ is a group.
	
	For converse, if $S$ is a group, then $\ell^{1}(S) $ has an identity. So ${\ell^{1}(S)}^{**}$ has an identity. It follows that $\ell^{1}(S)^{**}$ has an ultra central approximate identity.
\end{proof}
Let $\mathcal{A}$ and $\mathcal{B}$ be Banach algebras with $\varphi\in\Delta(\mathcal{B}).$ Consider the Cartesian product $\mathcal{A}\times \mathcal{B}$. Equip $A\times \mathcal{B}$ with the norm $||(a,b)||=||a||_{\mathcal{A}}+||b||_{\mathcal{B}}$ and $\varphi$-Lau product, that is, the product which is defined  by 
$$(a,b)(c,d)=(ac+\varphi(d)a+\varphi(b)c,bd),\quad (a,b,c,d\in \mathcal{A}).$$ Then $\mathcal{A}\times \mathcal{B}$ becomes a Banach algebra which we denote it with $\mathcal{A}\times_{\varphi}\mathcal{B}$. Note that $({\mathcal{A}\times_{\varphi}\mathcal{B}})^{**}$ (with the first Arens product) is isometrically isomorphism  with  $\mathcal{A}^{**}\times_{\varphi}\mathcal{B}^{**}$, and also for every $(m,n)$ and $(a,b)$ in  $\mathcal{A}^{**}\times_{\varphi}\mathcal{B}^{**}$ we have $$(m,n)(a,b)=(ma+n(\phi)a+b(\phi)m,nb),$$ for more information about the Lau product and its generalization see \cite{Lau:83} and \cite{San:07}.
\begin{Proposition}
	Let $\mathcal{A}$ and $\mathcal{B}$ be Banach algebras with $\varphi\in\Delta(\mathcal{B}).$ If  $\mathcal{A}\times_{\varphi}\mathcal{B}$ has an ultra central approximate identity, then $\mathcal{B}$ has an ultra central approximate identity.
\end{Proposition}
\begin{proof}
	Let $\mathcal{A}\times_{\varphi}\mathcal{B}$ has an ultra central approximate identity. Then there exists a net $(m_{\alpha},n_{\alpha})$ in
	$\mathcal{A}^{**}\times_{\varphi}\mathcal{B}^{**}$ such that $$(m_{\alpha},n_{\alpha})(a,b)=(a,b)(m_{\alpha},n_{\alpha}),\quad (a,b)(m_{\alpha},n_{\alpha})\rightarrow (a,b)\qquad (a\in \mathcal{A},b\in \mathcal{B}).$$
	Thus we have   $$(m_{\alpha}a+n_{\alpha}(\varphi)a+\varphi(b)m_{\alpha},n_{\alpha}b)=(am_{\alpha}+n_{\alpha}(\phi)a+\varphi(b)m_{\alpha},bn_{\alpha})\rightarrow (a,b),\quad (a\in \mathcal{A},b\in \mathcal{B}).$$ Therefore $n_{\alpha}b=bn_{\alpha}\rightarrow b$ for all $b\in \mathcal{B}.$ So $\mathcal{B}$ has an ultra central approximate identity.
\end{proof}
Let $\mathcal{A}$ be a Banach algebra and $X$ be a Banach $\mathcal{A}$-bimodule. The module extension $\mathcal{A}\oplus X$ is a Banach algebra with the following multiplication 
$$(a,x)(b,y)=(ab,ay+bx),\quad (a,b\in \mathcal{A},x,y\in X)$$
and the norm $||(a,x)||=||a||_{\mathcal{A}}+||x||_{X}.$ The Banach $\mathcal{A}$-bimodule $X$ is called commutative if $ax=xa$ for every $a\in \mathcal{A}$ and $x\in X$.
It is easy to see that $(\mathcal{A}\oplus X)^{**}$ can be identified with $\mathcal{A}^{**}\oplus_{1} X^{**}$ (as a Banach space) and the first Arens product on  $(\mathcal{A}\oplus X)^{**}$  is given by 
$$(m,\lambda)(n,\mu)=(mn,m\mu+\lambda n),\quad (m,n\in \mathcal{A}^{**}, \lambda, \mu \in X^{**}).$$ For more information about the module extension see \cite{zhang:02}.
\begin{Proposition}
	Let $\mathcal{A}$ be a Banach algebra and $X$ be a commutative Banach $\mathcal{A}$-bimodule. If  the module extension $\mathcal{A}\oplus X$ has an ultra central approximate identity, then $\mathcal{A}$ has an ultra central approximate identity.
\end{Proposition}
\begin{proof}
	Suppose that  the module extension $\mathcal{A}\oplus X$ has an ultra central approximate identity. Then there exists a net $(m_{\alpha},\lambda_{\alpha})$ in $(\mathcal{A}\oplus X)^{**}$ such that $$(m_{\alpha},\lambda_{\alpha})(a,x)=(a,x)(m_{\alpha},\lambda_{\alpha}),\quad (m_{\alpha},\lambda_{\alpha})(a,x)\rightarrow(a,x),\quad (a\in \mathcal{A}, x\in X).$$
	It implies that for every $a\in \mathcal{A}$ and $x\in X$,
	\begin{enumerate}
		\item [(i)]	$(m_{\alpha}a,m_{\alpha}x+\lambda_{\alpha}a)=(am_{\alpha},a\lambda_{\alpha}+xm_{\alpha})$,
		\item[(ii)] $(m_{\alpha}a,m_{\alpha}x+\lambda_{\alpha}a)\longrightarrow(a,x)$.
	\end{enumerate}
 So $m_{\alpha}a=am_{\alpha}$ and $m_{\alpha}a\rightarrow a$. Thus $\mathcal{A}$ has an ultra central approximate identity.
\end{proof}
\begin{small}
	
\end{small}
\end{document}